\documentclass[12pt]{amsart}
\usepackage{amsmath,amssymb,amsthm,amscd,amsfonts,geometry,color}
\theoremstyle{plain}
\newtheorem{thm}{Theorem}[section]
\newtheorem{prop}[thm]{Proposition}
\newtheorem{cor}[thm]{Corollary}
\newtheorem{lem}[thm]{Lemma}

\theoremstyle{remark}

\theoremstyle{definition}

\newcommand{\R}{\mathbb{R}}                       
\newcommand{\N}{\mathbb{N}}                       

\newcommand{\Q}{\mathbf{Q}}                       

\newcommand{\0}{\mathbf{0}}                       

\newcommand{\cl}{\operatorname{cl}}               

\newcommand{\uc}{\operatorname{UC}}             
\newcommand{\lip}{\operatorname{LIP}}             
\newcommand{\lipc}{\operatorname{lip}}            
\newcommand{\lp}{\operatorname{L^p}}            

\newcommand{\supp}{\operatorname{supp}}      

\begin{document}

\title[Characterizing compact sets in $\lp$-spaces and its application]{Characterizing compact sets in $\lp$-spaces and its application}
\author{Katsuhisa Koshino}
\address[Katsuhisa Koshino]{Faculty of Engineering, Kanagawa University, Yokohama, 221-8686, Japan}
\email{ft160229no@kanagawa-u.ac.jp}
\maketitle

\begin{abstract}
In this paper, we give a characterization of compact sets in $\lp$-spaces on metric measure spaces,
 which is a generalization of the Kolmogorov-Riesz theorem.
Using the criterion, we investigate the topological type of the space consisting of lipschitz maps with bounded supports.
\end{abstract}

\section{Introduction}

Characterizations of compact sets in function spaces, for instance, the Ascoli-Arzel\`a theorem, play important roles in the study of their topologies.
In this paper, we shall give a necessary and sufficient condition for subsets of $\lp$-spaces on metric measure spaces to be compact,
 and using it, we will decide the topology on the space consisting of lipschitz maps with bounded supports.
Let $X = (X,d,\mathcal{M},\mu)$ be a metric measure space,
 where $d$ is a metric on $X$,
 $\mathcal{M}$ is a $\sigma$-algebra of $X$,
 and $\mu$ is a measure on $\mathcal{M}$.
A measure space $X$ is \textit{Borel} if $\mathcal{M}$ contains the Borel sets of $X$.
A Borel measure space $X$ is called to be \textit{Borel-regular} if each $E \subset X$ is contained in a Borel set $B \subset X$ such that $\mu(B) = \mu(E)$.
Throughout this paper, $X$ is a Borel-regular Borel metric measure space such that every open ball with positive radius has a positive and finite measure.
Let $\R$ be the set of real numbers, $\N$ be the set of positive integers and $1 \leq p < \infty$.
We put
 $$\lp(X) = \bigg\{f : X \to \R \ \bigg| \ f \text{ is $\mathcal{M}$-measurable and } \int_X |f(x)|^p d\mu(x) < \infty\bigg\}$$
 with the following norm
 $$\|f\|_p = \bigg(\int_X |f(x)|^p d\mu(x)\bigg)^{1/p},$$
 where two functions that are coincident almost everywhere are identified.
The function space $\lp(X)$ is a Banach space, refer to \cite[Theorem~4.8]{Br}.

For a metric space $Y$, the symbol $B(y,r)$ stands for the closed ball centered at $y \in Y$ with radius $r > 0$.
Given a subset $E \subset X$, the characteristic function of $E$ is denoted by $\chi_E$,
 and more for a function $f : X \to \R$, $f\chi_E$ is defined by $f\chi_E(x) = f(x) \cdot \chi_E(x)$.
For each $f : \R^n \to \R$ and each $a \in \R^n$, where $n \in \N$,
 let $\tau_af$ be a function defined by $\tau_af(x) = f(x - a)$.
The following theorem proven by A.N.~Kolmogorov \cite{Kol} and M.~Riesz \cite{R} is useful for detecting compact sets in $\lp(\R^n)$,
 that is an $\lp$-version of the Ascoli-Arzel\`a theorem, refer to Theorem~5 of \cite{HH}.

\begin{thm}[Kolmogorov-Riesz]
A bounded set $F \subset \lp(\R^n)$, where $\R^n = (\R^n,d,\mathcal{M},\mu)$ is the Euclidean space with the Euclidean metric $d$, the Lebesgue measurable sets $\mathcal{M}$ and the Lebesgue measure $\mu$,
 is relatively compact if and only if the following conditions are satisfied.
\begin{enumerate}
 \item For each $\epsilon > 0$, there exists $\delta > 0$ such that $\|\tau_af - f\|_p < \epsilon$ for any $f \in F$ and $a \in \R^n$ with $|a| < \delta$.
 \item For each $\epsilon > 0$, there exists $r > 0$ such that $\|f\chi_{\R^n \setminus B(\0,r)}\|_p < \epsilon$ for any $f \in F$.
\end{enumerate}
\end{thm}

Fixing $f \in \lp(X)$ and $r > 0$, we can define the following function $A_rf : X \to \R$ by
 $$A_rf(x) = \frac{1}{\mu(B(x,r))}\int_{B(x,r)} f\chi_{B(x,r)}(y) d\mu(y),$$
 which is called \textit{the average function of $f$}.
It is known that average functions are $\mathcal{M}$-measurable, see \cite[Theorem~8.3]{Ye}.
A space $X$ is \textit{doubling} provided that the following condition is satisfied.
\begin{itemize}
 \item There exists $\gamma \geq 1$ such that $\mu(B(x,2r)) \leq \gamma\mu(B(x,r))$ for any point $x \in X$ and any positive number $r > 0$.
\end{itemize}
The above $\gamma$ is called \textit{the doubling constant}.
P.~G\'{o}rka and A.~Macios \cite{GM} generalized the Kolmogorov-Riesz theorem for doubling metric measure spaces as follows:

\begin{thm}
Let $X$ be a doubling metric measure space, $F \subset \lp(X)$ be a bounded subset and $p > 1$.
Suppose that
 $$\inf\{\mu(B(x,r)) \mid x \in X\} > 0$$
 for any $r > 0$.
Then $F$ is relatively compact if and only if the following hold.
\begin{enumerate}
 \item For every $\epsilon > 0$, there is $\delta > 0$ such that for each $f \in F$ and each $r \in (0,\delta)$, $\|A_rf - f\|_p < \epsilon$.
 \item For each $\epsilon > 0$, there exists a bounded set $E$ in $X$ such that $\|f\chi_{X \setminus E}\|_p < \epsilon$ for any $f \in F$.
\end{enumerate}
\end{thm}

In the proof of the above theorem, the Hardy-Littlewood maximal inequation, refer to \cite[Theorem~2.2]{Hein}, was applied.
This inequation does not hold when $p = 1$,
 so it is assumed that $p > 1$ in the above theorem.
On the other hand, the average function $A_rf$ of $f \in \lp(X)$ can be estimated by $f$, see Lemma~\ref{est.}.
Combining this with the natural condition between the metric and measure of $X$,
 which is used to give a criterion for subsets of Banach function spaces to be compact in \cite{GR}, but without the condition that $p \neq 1$, we shall establish the following theorem.

\begin{thm}\label{cpt.}
Let $X$ be a doubling metric measure space and $F$ be a bounded subset of $\lp(X)$.
Suppose that for any $x \in X$ and any $r > 0$,
 $$\mu(B(x,r) \triangle B(y,r)) \to 0$$
 as $y \to x$.\footnote{For subsets $A, B \subset X$, let $A \triangle B = (A \setminus B) \cup (B \setminus A)$.
 This continuous property of the measure $\mu$ with respect to the metric $d$ is investigated in \cite{GG}.}
Then $F$ is relatively compact if and only if the following are satisfied.
\begin{enumerate}
 \item For every $\epsilon > 0$, there is $\delta > 0$ such that for each $f \in F$ and each $r \in (0,\delta)$, $\|A_rf - f\|_p < \epsilon$.
 \item For each $\epsilon > 0$, there exists a bounded set $E$ in $X$ such that $\|f\chi_{X \setminus E}\|_p < \epsilon$ for any $f \in F$.
\end{enumerate}
\end{thm}

In the theory of infinite-dimensional topology, it plays an important role to recognize typical subspaces of the separable Hilbert space $\ell_2$ and the Hilbert cube $\Q$ among function spaces.
Under our assumption, the separability of $\lp(X)$ follows from the one of $X$, see \cite[Proposition~3.4.5]{Coh}.
Moreover, as is easily observed,
 when $X$ is infinite,
 $\lp(X)$ is infinite-dimensional.
In the following theorem, the latter part follows from the efforts by R.D.~Anderson \cite{Ande1} and M.I.~Kadec \cite{Kad}.

\begin{thm}\label{Lp}
Let $X$ be an infinite separable metric measure space.
Then $\lp(X)$ is an infinite-dimensional separable Banach space,
 so it is homeomorphic to $\ell_2$.
\end{thm}

Consider the topologies on subspaces of $\lp(X)$.
The following subspace of $\ell_2$,
 $$\ell_2^f = \{(x(n)) \in \ell_2 \mid x(n) = 0 \text{ for almost all } n \in \N\},$$
 is detected among several function spaces as a factor.
For example, the author \cite{Kos14} generalized R.~Cauty's result \cite{Ca} and recognized the topological type of the subspace
 $$\uc(X) = \{f \in \lp(X) \mid f \text{ is uniformly continuous}\}$$ as follows:\footnote{Let $P$ be a property of functions.
 A function $g \in \{f \in \lp(X) \mid f \text{ satisfies } P\}$ if there is $f : X \to \R$ such that $f$ satisfies $P$ and $g = f$ almost everywhere.}

\begin{thm}
Let $X$ be a separable and locally compact metric measure space such that $\{x \in X \mid \mu(\{x\}) \neq 0\}$ is not dense in $X$.
Then $\uc(X)$ is homeomorphic to $(\ell_2^f)^\N$.
\end{thm}

Let $\lip(X)$ be the subspace in $\lp(X)$ consisting of lipschitz maps.
In this paper, we shall study the following subspace:
 $$\lip_b(X) = \{f \in \lip(X) \mid f \text{ has a bounded support}\}.$$
As an application of Theorem~\ref{cpt.}, we show the following:

\begin{cor}\label{lip.}
Let $X$ be a metric measure space satisfying the following conditions:
 {\rm (1)} $X$ is non-degenerate and separable;
 {\rm (2)} $X$ is doubling;
 {\rm (3)} for each point $x \in X$, the function
 $$(0,\infty) \ni r \mapsto \mu(B(x,r)) \in (0,\infty)$$
 is continuous.
Then the pair $(\lp(X),\lip_b(X))$ is homeomorphic to $(\ell_2 \times \Q,\ell_2^f \times \Q)$.\footnote{For spaces $Y_1 \supset Y_2$ and $Z_1 \supset Z_2$, the pair $(Y_1,Y_2)$ is homeomorphic to $(Z_1,Z_2)$ if there is a homeomorphism $f : Y_1 \to Z_1$ such that $f(Y_2) = Z_2$.}
\end{cor}

\section{Proof of Theorem~\ref{cpt.}}

For any function $f, g \in \lp(X)$, any point $x \in X$, and any positive number $r > 0$, observe that
\begin{align*}
 A_rf(x) - A_rg(x) &= \frac{1}{\mu(B(x,r))}\int_{B(x,r)} f(y) d\mu(y) - \frac{1}{\mu(B(x,r))}\int_{B(x,r)} g(y) d\mu(y)\\
 &= \frac{1}{\mu(B(x,r))}\int_{B(x,r)} (f(y) - g(y)) d\mu(y) = A_r(f(x) - g(x)).
\end{align*}
For $p \in [1,\infty)$, denote the conjugate of $p$ by $q$,
 that is, $1/p + 1/q = 1$.
Due to H\"{o}lder's inequality,
\begin{align*}
 |A_rf(x)| &= \bigg|\frac{1}{\mu(B(x,r))}\int_{B(x,r)} f(y) d\mu(y)\bigg| \leq \frac{1}{\mu(B(x,r))}\int_{B(x,r)} |f(y)| d\mu(y)\\
 &\leq \frac{1}{\mu(B(x,r))}\bigg(\int_{B(x,r)} |f(y)|^p d\mu(y)\bigg)^{1/p}\bigg(\int_{B(x,r)} d\mu(x)\bigg)^{1/q}\\
 &= \frac{1}{\mu(B(x,r))}\bigg(\int_X |f\chi_{B(x,r)}(y)|^p d\mu(y)\bigg)^{1/p}(\mu(B(x,r)))^{1/q}\\
 &= \|f\chi_{B(x,r)}\|_p(\mu(B(x,r)))^{-1/p}.
\end{align*}
Average functions can be uniformly bounded as follows:

\begin{lem}\label{est.}
Let $X$ be a doubling metric measure space with the doubling constant $\gamma$ and $f \in \lp(X)$.
For each $r > 0$, $\|A_rf\|_p \leq \gamma^{1/p}\|f\|_p$.
\end{lem}

\begin{proof}
Since $X$ is doubling,
 for any $z, w \in X$, if $z \in B(w,r)$,
 then
 $$\mu(B(w,r)) \leq \mu(B(z,2r)) \leq \gamma\mu(B(z,r)),$$
 and hence $1/\mu(B(z,r)) \leq \gamma/\mu(B(w,r))$.
Applying the Fubini-Tonelli theorem, we have
\begin{align*}
 \|A_rf\|_p^p &= \int_X |A_rf(x)|^p d\mu(x) \leq \int_X \bigg(\bigg(\int_X |f\chi_{B(x,r)}(y)|^p d\mu(y)\bigg)^{1/p}(\mu(B(x,r)))^{-1/p}\bigg)^p d\mu(x)\\
 &= \int_X \frac{1}{\mu(B(x,r))} \bigg(\int_X |f\chi_{B(x,r)}(y)|^p d\mu(y)\bigg) d\mu(x)\\
 &= \int_X |f(y)|^p \bigg(\int_X \frac{\chi_{B(y,r)}(x)}{\mu(B(x,r))} d\mu(x)\bigg) d\mu(y)\\
 &\leq \int_X |f(y)|^p \bigg(\int_X \frac{\gamma\chi_{B(y,r)}(x)}{\mu(B(y,r))} d\mu(x)\bigg) d\mu(y)\\
 &= \int_X \frac{|f(y)|^p\gamma\mu(B(y,r))}{\mu(B(y,r))} d\mu(y) = \gamma\|f\|_p^p.
\end{align*}
Consequently, $\|A_rf\|_p \leq \gamma^{1/p}\|f\|_p$.
\end{proof}

According to the above lemma, the condition~(1) of Theorem~\ref{cpt.} is shown to be necessary for a subset in $\lp(X)$ being relatively compact.

\begin{prop}\label{approx.}
Let $X$ be a doubling metric measure space.
If a subset $F$ is relatively compact in $\lp(X)$,
 then the following holds.
\begin{itemize}
 \item For any $\epsilon > 0$, there is $\delta > 0$ such that for each $f \in F$ and each $r \in (0,\delta)$, $\|A_rf - f\|_p < \epsilon$.
\end{itemize}
\end{prop}

\begin{proof}
Since $F$ is relatively compact,
 and hence it is totally bounded,
 we can find finitely many functions $f_i \in \lp(X)$, $1 \leq i \leq n$, so that for each $f \in F$, there exists $i \in \{1, \cdots, n\}$ such that $\|f - f_i\|_p \leq \epsilon\gamma^{-1/p}/3$,
 where $\gamma$ is the doubling constant of $X$.
Here we may assume that $|f_i| \leq c_i$ for some $c_i > 0$ and $f_i$ has a bounded support $E_i \subset X$.
Using Lemma~\ref{est.}, we get
 $$\|A_rf - A_rf_i\|_p = \|A_r(f - f_i)\|_p \leq \gamma^{1/p}\|f - f_i\|_p < \epsilon/3.$$
By Lebesgue's differentiation theorem, see \cite[Theorem~1.8]{Hein},
 for each $i \in \{1, \cdots, n\}$, $A_rf_i \to f_i$ as $r \to 0$ almost everywhere.
Set $E = \bigcup_{x \in \bigcup_{i = 1}^n E_i} B(x,1)$,
 that is a bounded subset of $X$,
 and hence $\mu(E) < \infty$.
Fixing any $r \in (0,1)$, observe that $f_i(x) = A_rf_i(x) = 0$ for all $x \in X \setminus E$,
 and
 \begin{align*}
  |A_rf_i(x)| &= \bigg|\frac{1}{\mu(B(x,r))}\int_{B(x,r)} f_i(y) d\mu(y)\bigg| \leq \frac{1}{\mu(B(x,r))}\int_{B(x,r)} |f_i(y)| d\mu(y)\\
  &\leq \frac{1}{\mu(B(x,r))}\int_{B(x,r)} c_i d\mu(y) = \frac{c_i\mu(B(x,r))}{\mu(B(x,r))} = c_i
 \end{align*}
 for all $x \in E$.
Applying the dominated convergence theorem, we can obtain that $\|A_rf_i - f_i\|_p \to 0$ as $r \to 0$.
So there is $\delta_i \in (0,1)$ such that if $r < \delta_i$,
 then $\|A_rf_i - f_i\|_p < \epsilon/3$.
Let $\delta = \min\{\delta_i \mid 1 \leq i \leq n\}$,
 so for any $0 < r < \delta$,
 $$\|A_rf - f\|_p \leq \|A_rf - A_rf_i\|_p + \|A_rf_i - f_i\|_p + \|f_i - f\|_p < \epsilon.$$
Thus the proof is completed.
\end{proof}

Measures of closed balls centered at points in a bounded subset with the same radius of a doubling metric measure space are lower bounded.

\begin{lem}\label{ball}
If $X$ is doubling and $E \subset X$ is bounded,
 then for each $r > 0$,
 $$\inf\{\mu(B(x,r)) \mid x \in E\} > 0.$$
\end{lem}

\begin{proof}
Let $R < \infty$ be the diameter of $E$,
 so for any $x \in E$, $E \subset B(x,R)$,
 which implies that $\mu(E) \leq \mu(B(x,R))$.
Taking $n \in \N$ such that $R \leq 2^nr$, since $X$ is doubling,
 we have that
 $$\mu(B(x,r)) \geq \gamma^{-n}\mu(B(x,2^nr)) \geq \gamma^{-n}\mu(B(x,R)) \geq \gamma^{-n}\mu(E),$$
 where $\gamma$ is the doubling constant of $X$.
\end{proof}

Fix a point $x \in X$ and a positive number $r > 0$.
If $\mu(B(x,r) \triangle B(y,r)) \to 0$ as a point $y \in X$ tends to $x$,
 then $\mu(B(y,r)) \to \mu(B(x,r))$ because
 \begin{align*}
  |\mu(B(x,r)) - \mu(B(y,r))| &= |\mu(B(x,r) \setminus B(y,r)) - \mu(B(y,r) \setminus B(x,r))|\\
  &\leq \mu(B(x,r) \setminus B(y,r)) + \mu(B(y,r) \setminus B(x,r)) = \mu(B(x,r) \triangle B(y,r)).
 \end{align*}
We have the following proposition.

\begin{prop}\label{equiconti.}
Let $x \in X$ and $r > 0$.
Suppose that $F$ is a bounded set in $\lp(X)$ and that $\mu(B(x,r) \triangle B(y,r)) \to 0$ as $y \to x$.
If $p > 1$,
 then the following holds.
\begin{itemize}
 \item For each $\epsilon > 0$, there exists $\delta > 0$ such that if $d(x,y) < \delta$,
 then $|A_rf(x) - A_rf(y)| < \epsilon$ for any $f \in F$.
\end{itemize}
Additionally, if $X$ is doubling,
 and for each $\lambda > 0$,
 there exists $\sigma > 0$ such that for any $f \in F$ and any $s \in (0,\sigma)$, $\|A_sf - f\|_1 < \lambda$,
 then the above holds even if $p = 1$.
\end{prop}

\begin{proof}
Since $F$ is bounded,
 there is $c_1 > 0$ such that $\|f\|_p < c_1$ for every $f \in F$.
Let $c_2 = \mu(B(x,r))$.
When $p = 1$, by virtue of the additional hypothesis, we can find $s > 0$ such that $\|A_sf - f\|_1 < \epsilon c_2/4$ for every $f \in F$,
 and put $c_3 = \inf\{\mu(B(z,s)) \mid z \in B(x,r + 1)\} > 0$ due to Lemma~\ref{ball}.
By the assumption, there is $\delta \in (0,1)$ such that if $d(x,y) < \delta$,
 then
 $$\mu(B(x,r) \triangle B(y,r)) < \left\{
 \begin{array}{ll}
  \big(\frac{\epsilon c_2}{2c_1}\big)^q &\text{if } p > 1,\\
  \frac{\epsilon c_2c_3}{4c_1} &\text{if } p = 1.
 \end{array}
 \right.$$
Moreover, we may suppose that $\mu(B(y,r)) \leq 2c_2$ and $|1/\mu(B(x,r)) - 1/\mu(B(y,r))| < \epsilon/(2c_1(2c_2)^{1/q})$ because $\mu(B(y,r)) \to \mu(B(x,r))$ as $y \to x$.
To show that $\delta$ is desired, let $y \in X$ with $d(x,y) < \delta$.
As is easily observed,
 $$\bigg|\int_{B(x,r)} f(z) d\mu(z) - \int_{B(y,r)} f(z) d\mu(z)\bigg| \leq \int_{B(x,r) \triangle B(y,r)} |f(z)| d\mu(z).$$
Therefore due to H\"{o}lder's inequality,
\begin{align*}
 |A_rf(x) - A_rf(y)| &= \bigg|\frac{1}{\mu(B(x,r))}\int_{B(x,r)} f(z) d\mu(z) - \frac{1}{\mu(B(y,r))}\int_{B(y,r)} f(z) d\mu(z)\bigg|\\
 &\leq \frac{1}{\mu(B(x,r))}\bigg|\int_{B(x,r)} f(z) d\mu(z) - \int_{B(y,r)} f(z) d\mu(z)\bigg|\\
 &\ \ \ \ \ \ \ \ + \bigg|\frac{1}{\mu(B(x,r))} - \frac{1}{\mu(B(y,r))}\bigg|\bigg|\int_{B(y,r)} f(z) d\mu(z)\bigg|\\
 &\leq \frac{1}{\mu(B(x,r))}\int_{B(x,r) \triangle B(y,r)} |f(z)| d\mu(z)\\
 &\ \ \ \ \ \ \ \ + \bigg|\frac{1}{\mu(B(x,r))} - \frac{1}{\mu(B(y,r))}\bigg|\bigg(\int_{B(y,r)} |f(z)| d\mu(z)\bigg)\\
 &\leq \frac{1}{\mu(B(x,r))}\bigg(\int_{B(x,r) \triangle B(y,r)} |f(z)|^p d\mu(z)\bigg)^{1/p}\bigg(\int_{B(x,r) \triangle B(y,r)} d\mu(z)\bigg)^{1/q}\\
 &\ \ \ \ \ \ \ \ + \bigg|\frac{1}{\mu(B(x,r))} - \frac{1}{\mu(B(y,r))}\bigg|\bigg(\int_{B(y,r)} |f(z)|^p d\mu(z)\bigg)^{1/p}\bigg(\int_{B(y,r)} d\mu(z)\bigg)^{1/q}\\
 &\leq \frac{1}{\mu(B(x,r))}\bigg(\int_{B(x,r) \triangle B(y,r)} |f(z)|^p d\mu(z)\bigg)^{1/p}\bigg(\int_{B(x,r) \triangle B(y,r)} d\mu(z)\bigg)^{1/q}\\
 &\ \ \ \ \ \ \ \ + \bigg|\frac{1}{\mu(B(x,r))} - \frac{1}{\mu(B(y,r))}\bigg|\|f\|_p(\mu(B(y,r)))^{1/q}\\
 &< \frac{1}{\mu(B(x,r))}\bigg(\int_{B(x,r) \triangle B(y,r)} |f(z)|^p d\mu(z)\bigg)^{1/p}\bigg(\int_{B(x,r) \triangle B(y,r)} d\mu(z)\bigg)^{1/q}\\
 &\ \ \ \ \ \ \ \ + \frac{\epsilon}{2c_1(2c_2)^{1/q}} \cdot c_1(2c_2)^{1/q}\\
 &< \frac{1}{\mu(B(x,r))}\bigg(\int_{B(x,r) \triangle B(y,r)} |f(z)|^p d\mu(z)\bigg)^{1/p}\bigg(\int_{B(x,r) \triangle B(y,r)} d\mu(z)\bigg)^{1/q} + \frac{\epsilon}{2}.
\end{align*}
In the case where $p > 1$,
\begin{multline*}
 \frac{1}{\mu(B(x,r))}\bigg(\int_{B(x,r) \triangle B(y,r)} |f(z)|^p d\mu(z)\bigg)^{1/p}\bigg(\int_{B(x,r) \triangle B(y,r)} d\mu(z)\bigg)^{1/q}\\
 \leq \frac{1}{\mu(B(x,r))}\|f\|_p(\mu(B(x,r) \triangle B(y,r))^{1/q} < \frac{1}{c_2} \cdot c_1 \cdot \frac{\epsilon c_2}{2c_1} = \frac{\epsilon}{2}.
\end{multline*}
In the case where $p = 1$,
\begin{align*}
 \int_{B(x,r) \triangle B(y,r)} |f(z)| d\mu(z) & \leq \int_{B(x,r) \triangle B(y,r)} |f(z) - A_sf(z)| d\mu(z) + \int_{B(x,r) \triangle B(y,r)} |A_sf(z)| d\mu(z)\\
 &\leq \|A_sf - f\|_1 + \int_{B(x,r) \triangle B(y,r)} \bigg(\frac{1}{\mu(B(z,s))}\int_{B(z,s)} |f(w)| d\mu(w)\bigg) d\mu(z)\\
 &\leq \|A_sf - f\|_1 + \int_{B(x,r) \triangle B(y,r)} \frac{1}{c_3}\|f\|_1 d\mu(z)\\
 &\leq \|A_sf - f\|_1 + \frac{1}{c_3}\|f\|_1\mu(B(x,r) \triangle B(y,r)) < \frac{\epsilon c_2}{4} + \frac{1}{c_3} \cdot c_1 \cdot \frac{\epsilon c_2c_3}{4c_1} = \frac{\epsilon c_2}{2},
\end{align*}
and hence
 $$\frac{1}{\mu(B(x,r))}\bigg(\int_{B(x,r) \triangle B(y,r)} |f(z)|^p d\mu(z)\bigg)^{1/p}\bigg(\int_{B(x,r) \triangle B(y,r)} d\mu(z)\bigg)^{1/q} \leq \frac{1}{c_2} \cdot \frac{\epsilon c_2}{2} = \frac{\epsilon}{2}.$$
It follows that $|A_rf(x) - A_rf(y)| < \epsilon$.
We complete the proof.
\end{proof}

The next lemma will be used in the proof of Proposition~\ref{av.rel.cpt.}.

\begin{lem}\label{rel.cpt.}
Let $x \in X$ and $r > 0$.
Suppose that $F$ is a bounded subset of $\lp(X)$.
Then the set $\{A_rf(x) \mid f \in F\}$ is relatively compact in $\R$.
\end{lem}

\begin{proof}
Because $F$ is bounded,
 we can choose $c > 0$ so that $\|f\|_p < c$ for every $f \in F$.
Then
 $$|A_rf(x)| \leq \|f\chi_{B(x,r)}\|_p(\mu(B(x,r)))^{-1/p} < c(\mu(B(x,r)))^{-1/p}.$$
Hence $\{A_rf(x) \mid f \in F\}$ is bounded in $\R$,
 so it is relatively compact.
\end{proof}

The following proposition is a key ingredient in proving the ``if'' part of Theorem~\ref{cpt.}.

\begin{prop}\label{av.rel.cpt.}
Let $X$ be a doubling metric measure space, $F \subset \lp(X)$ be a bounded subset and $r > 0$.
Suppose that $E \subset X$ is bounded and that for each $x \in X$, $\mu(B(x,r) \triangle B(y,r)) \to 0$ as $y$ tends to $x$.
Then $\{(A_rf)\chi_E \mid f \in F\}$ is relatively compact in $\lp(X)$ when $p > 1$.
Additionally, if for every $\epsilon > 0$,
 there is $\delta > 0$ such that for any $f \in F$ and any $r \in (0,\delta)$, $\|A_rf - f\|_1 < \epsilon$,
 then the above is valid even if $p = 1$.
\end{prop}

\begin{proof}
We need only to show that $\{(A_rf)\chi_E \mid f \in F\}$ is totally bounded in $\lp(X)$.
To prove it, fix any $\epsilon$.
Recall that $\mu(E) < \infty$.
Because $F$ is bonded,
 there is $c > 0$ such that $\|f\|_p < c$ for all $f \in F$.
According to Lemma~\ref{ball}, put $a = \inf\{\mu(B(x,r)) \mid x \in E\} > 0$.
By Proposition~\ref{equiconti.}, for each $x \in X$, there is $\delta(x) \in (0,1)$ such that if $d(x,y) < \delta(x)$,
 then $|A_rf(x) - A_rf(y)| < \epsilon(\mu(E))^{-1/p}/4$.
Using Vitali's covering theorem, refer to \cite[Theorem~6.20]{Ye},
 we can obtain finitely many points $x_i \in E$ and positive numbers $0 < r_i < \delta(x_i)$, $i = 1, \cdots, n$ so that $\mu(E \setminus \bigcup_{i = 1}^n B(x_i,r_i)) < a\epsilon^p/(2c)^p$ and $\{B(x_i,r_i) \mid i = 1, \cdots, n\}$ is pairwise disjoint.

Due to Lemma~\ref{rel.cpt.}, for each $i \in \{1, \cdots, n\}$, the subset $\{A_rf(x_i) \mid f \in F\}$ is totally bounded in $\R$,
 and hence we can choose $a_{(i,j(i))} \in \R$, $j(i) = 1, \cdots, m(i)$, so that
 $$\{A_rf(x_i) \mid f \in F\} \subset \bigcup_{j(i) = 1}^{m(i)} B(a_{(i,j(i))},\epsilon(\mu(E))^{-1/p}/4).$$
Now let
 $$K = \Bigg\{\sum_{i = 1}^n a_{(i,j(i))}\chi_{B(x_i,r_i) \cap E} \ \Bigg| \ \text{for each } i \in \{1, \cdots, n\}, j(i) \in \{1, \cdots, m(i)\}\Bigg\}$$
 be a finitely many collection of simple functions.
It remains to show that $\{(A_rf)\chi_E \mid f \in F\} \subset \bigcup_{\phi \in K} B(\phi,\epsilon)$.
Take any $f \in F$.
For every $i = 1, \cdots, n$, there is $j(i)$ such that $|A_rf(x_i) - a_{(i,j(i))}| < \epsilon(\mu(E))^{-1/p}/4$.
Since $r_i < \delta(x_i)$,
 for any $x \in B(x_i,r_i)$,
 $$|A_rf(x) - a_{(i,j(i))}| \leq |A_rf(x) - A_rf(x_i)| + |A_rf(x_i) - a_{(i,j(i))}| < \frac{\epsilon(\mu(E))^{-1/p}}{2}.$$
Therefore we have that
\begin{align*}
 \Bigg\|\sum_{i = 1}^n (A_rf)\chi_{B(x_i,r_i) \cap E} - \sum_{i = 1}^n a_{(i,j(i))}\chi_{B(x_i,r_i) \cap E}\Bigg\|_p &= \Bigg(\sum_{i = 1}^n \int_{B(x_i,r_i) \cap E} |A_rf(x) - a_{(i,j(i))}|^p d\mu(x)\Bigg)^{1/p}\\
 &\leq \Bigg(\sum_{i = 1}^n \int_{B(x_i,r_i) \cap E} \bigg(\frac{\epsilon(\mu(E))^{-1/p}}{2}\bigg)^p d\mu(x)\Bigg)^{1/p}\\
 &= \bigg(\frac{\epsilon(\mu(E))^{-1/p}}{2}\bigg)\Bigg(\sum_{i = 1}^n \mu(B(x_i,r_i) \cap E)\Bigg)^{1/p}\\
 &\leq \bigg(\frac{\epsilon(\mu(E))^{-1/p}}{2}\bigg)(\mu(E))^{1/p} = \frac{\epsilon}{2}.
\end{align*}
Set $A = E \setminus \bigcup_{i = 1}^n B(x_i,r_i)$.
Observe that
\begin{align*}
 \Bigg\|(A_rf)\chi_E - \sum_{i = 1}^n (A_rf)\chi_{B(x_i,r_i) \cap E}\Bigg\|_p &= \bigg(\int_A |A_rf(x)|^p d\mu(x)\bigg)^{1/p}\\
 &\leq \bigg(\int_A (\|f\chi_{B(x,r)}\|_p(\mu(B(x,r)))^{-1/p})^p d\mu(x)\bigg)^{1/p}\\
 &\leq \|f\|_p\bigg(\int_A \frac{1}{\mu(B(x,r))} d\mu(x)\bigg)^{1/p} \leq c\bigg(\frac{\mu(A)}{a}\bigg)^{1/p} < \frac{\epsilon}{2}.
\end{align*}
It follows that
\begin{align*}
 &\Bigg\|(A_rf)\chi_E - \sum_{i = 1}^n a_{(i,j(i))}\chi_{B(x_i,r_i) \cap E}\Bigg\|_p\\
 &\ \ \ \ \leq \Bigg\|(A_rf)\chi_E - \sum_{i = 1}^n (A_rf)\chi_{B(x_i,r_i) \cap E}\Bigg\|_p + \Bigg\|\sum_{i = 1}^n (A_rf)\chi_{B(x_i,r_i) \cap E} - \sum_{i = 1}^n a_{(i,j(i))}\chi_{B(x_i,r_i) \cap E}\Bigg\|_p < \epsilon.
\end{align*}
The proof is completed.
\end{proof}

Now we shall prove Theorem~\ref{cpt.}.

\begin{proof}[Proof of Theorem~\ref{cpt.}]
Firstly, we will show the ``only if'' part.
The condition~(1) follows from Proposition~\ref{approx.}.
It remains to show (2).
Fix any $\epsilon > 0$.
Since $F$ is relatively compact, so totally bounded,
 we can obtain finitely many functions $f_i \in \lp(X)$, $1 \leq i \leq n$, with a bounded support $E_i$ so that for each $f \in F$, there is $i \in \{1, \cdots, n\}$ such that $\|f - f_i\|_p \leq \epsilon$.
Then the union $E = \bigcup_{i = 1}^n E_i$ is the desired bounded set.
Indeed,
 $$\|f\chi_{X \setminus E}\|_p \leq \|f\chi_{X \setminus E} - f_i\chi_{X \setminus E}\|_p + \|f_i\chi_{X \setminus E}\|_p \leq \|f - f_i\|_p \leq \epsilon,$$
 which means that (2) holds.

Next, we prove the ``if'' part.
We shall show that $F$ is totally bounded.
To prove it, fix any $\epsilon > 0$.
By the condition~(2), we can find a bounded subset $E \subset X$ such that $\|f\chi_{X \setminus E}\|_p < \epsilon/2$ for all $f \in F$.
Observe that there is $\delta > 0$ such that for each $f \in F$ and each $r \in (0,\delta)$, $\|f - A_rf\|_p < \epsilon/2$ by (1),
 and hence we get
 $$\|f - (A_rf)\chi_E\|_p \leq \|f\chi_E - (A_rf)\chi_E\|_p + \|f\chi_{X \setminus E}\|_p \leq \|f - A_rf\|_p + \|f\chi_{X \setminus E}\|_p < \epsilon.$$
On the other hand, $\{(A_rf)\chi_E \mid f \in F\}$ is totally bounded in $\lp(X)$ due to Proposition~\ref{av.rel.cpt.}.
This means that $F$ is relatively compact.
\end{proof}

\section{Topological properties of the space $\lip_b(X)$}

In this section, we shall prove that $\lip_b(X) \subset \lp(X)$ is a $\sigma$-compact subset containing topological copies of $\Q$ under our assumption.
From now on, assume that $X$ is non-empty and fix any point $x_0 \in X$.
Given a positive integer $n \in \N$, set
 $$L(n) = \{f \in \lip(X) \mid \|f\|_p \leq n, \lipc{f} \leq n \text{ and } \supp{f} \subset B(x_0,n)\},$$
 where $\lipc{f}$ is the lipschitz constant of $f$ and $\supp{f}$ is the support of $f$.
As is easily observed,
 the space $\lip_b(X) = \bigcup_{n \in \N} L(n)$.
Using Theorem~\ref{cpt.}, we can prove the following:

\begin{prop}\label{sigma-cpt.}
Let $X$ be a doubling metric measure space such that for each point $x \in X$ and each positive number $r > 0$, $\mu(B(x,r) \triangle B(y,r)) \to 0$ as $y \to x$.
Then each $L(n)$ is compact,
 and hence $\lip_b(X)$ is $\sigma$-compact.
\end{prop}

\begin{proof}
Let $n \in \N$.
To prove that $L(n)$ is closed in $\lp(X)$, fix any sequence $\{f_k\}$ in $L(n)$ convergeing to $f \in \lp(X)$.
We will verify that $f \in L(n)$.
Observe that $\|f\|_p \leq n$ and $\supp{f} \subset B(x_0,n)$.
Since $\|f - f_k\|_p \to 0$,
 we can replace $\{f_k\}$ with a subsequence so that $f_k \to f$ almost everywhere due to Theorem~4.9 of \cite{Br}.
Then there is $E \subset X$ such that $\mu(E) = 0$ and $f_k(x) \to f(x)$ for every $x \in X \setminus E$.
Because each $\lipc{f_k} \leq n$,
 for any $x, y \in X \setminus E$, $|f_k(x) - f_k(y)| \leq nd(x,y)$,
 which implies that $|f(x) - f(y)| \leq nd(x,y)$.
Hence $f|_{X \setminus E}$ is an $n$-lipschitz map,
 so it is a uniformly continuous map.
Since $\mu(B(x,r)) > 0$ for any $x \in X$ and any $r > 0$,
 $X \setminus E$ is dense in $X$.
It follows that the restriction $f|_{X \setminus E}$ can be extended to a uniformly continuous map $\tilde{f}$.
Then $\tilde{f}$ is an $n$-lipschitz map.
Indeed, take any $x, y \in X$.
For each $\epsilon > 0$, we can choose $x', y' \in X \setminus E$ so that $d(x',y') \leq d(x,y) + \epsilon/(3n)$, $|\tilde{f}(x) - \tilde{f}(x')| \leq \epsilon/3$, and $|\tilde{f}(y) - \tilde{f}(y')| \leq \epsilon/3$.
Then
\begin{align*}
 |\tilde{f}(x) - \tilde{f}(y)| &\leq |\tilde{f}(x) - \tilde{f}(x')| + |\tilde{f}(x') - \tilde{f}(y')| + |\tilde{f}(y') - \tilde{f}(y)|\\
 &\leq \frac{\epsilon}{3} + nd(x',y') + \frac{\epsilon}{3} \leq nd(x,y) + \epsilon.
\end{align*}
Therefore $|\tilde{f}(x) - \tilde{f}(y)| \leq nd(x,y)$,
 and hence $\tilde{f}$ is $n$-lipschitz.
Note that $\tilde{f}$ is coincident with $f$.
Thus $f \in L(n)$.

By the definition of $L(n)$, it is bounded.
Since for every $f \in L(n)$, the support $\supp{f} \subset B(x_0,n)$,
 $L(n)$ satisfies the condition~(2) of Theorem~\ref{cpt.}.
It remains to show that the condition~(1) holds.
For each $0 < \epsilon < 1$, let $\delta = \epsilon/(n\mu(B(x_0,n + 1))^{1/p})$.
Taking any $f \in L(n)$ and any $r \in (0,\delta)$, since $f$ is $n$-lipschitz,
 we have that
 \begin{align*}
  |A_rf(x) - f(x)| &= \bigg|\frac{1}{\mu(B(x,r))}\int_{B(x,r)} f\chi_{B(x,r)}(y) d\mu(y) - f(x)\bigg|\\
  &\leq \frac{1}{\mu(B(x,r))}\int_{B(x,r)} |f(y) - f(x)| d\mu(y) \leq \frac{1}{\mu(B(x,r))}\int_{B(x,r)} nd(x,y) d\mu(y)\\
  &\leq nr < n\delta = \frac{\epsilon}{\mu(B(x_0,n + 1))^{1/p}}
 \end{align*}
 for each point $x \in X$.
Remark that $\supp{A_rf} \subset B(x_0,n + 1)$.
Then
\begin{align*}
 \|A_rf - f\|_p &= \bigg(\int_{B(x_0,n + 1)} |A_rf(x) - f(x)|^p d\mu(x)\bigg)^{1/p}\\
 &< \bigg(\int_{B(x_0,n + 1)} \bigg(\frac{\epsilon}{\mu(B(x_0,n + 1))^{1/p}}\bigg)^p d\mu(x)\bigg)^{1/p} = \epsilon.
\end{align*}
By virtue of Theorem~\ref{cpt.}, $L(n)$ is compact.
\end{proof}

Next we will show that each $L(n) \subset \lp(X)$ is homeomorphic to $\Q$ by using the following characterization \cite{Tor4}.

\begin{thm}[H.~Toru\'nczyk]\label{char.Q}
A compact AR $Y$ is homeomorphic to $\Q$ if and only if the following condition holds.
\begin{itemize}
 \item[$(\ast)$] For each continuous map $f : \Q \times \{0,1\} \to Y$ and each $\epsilon > 0$, there exists a continuous map $g : \Q \times \{0,1\} \to Y$ such that $g$ is $\epsilon$-close to $f$ and $g(\Q \times \{0\}) \cap g(\Q \times \{1\}) = \emptyset$.
\end{itemize}
\end{thm}

Recall that for a metric space $Y = (Y,d_Y)$, for functions $f : Z \to Y$ and $g : Z \to Y$, and for a positive number $\epsilon > 0$, $f$ is \textit{$\epsilon$-close} to $g$ provided that for each $z \in Z$, $d_Y(f(z),g(z)) \leq \epsilon$.
It follows from the following lemma that $L(n)$ satisfies condition~$(\ast)$ of Theorem~\ref{char.Q} immediately.

\begin{lem}\label{dcp}
Let $X$ be a doubling metric measure space with $\mu(\{x_0\}) = 0$.
For any $n \in \N$ and $0 < \epsilon < 1$, there exist maps $\Phi : L(n) \to \{f \in L(n) \mid \lipc{f} < n\}$ and $\Psi : L(n) \to \{f \in L(n) \mid \lipc{f} = n\}$ which are $\epsilon$-close to the identity map on $L(n)$.
\end{lem}

\begin{proof}
First we shall construct the map $\Phi$.
Let $\Phi(f) = (1 - \epsilon/n)f$ for each $f \in L(n)$.
Obviously $\supp{\Phi(f)} \subset B(x_0,n)$.
Since $f$ is $n$-lipschitz,
 for any $x, y \in X$,
 \begin{align*}
  |\Phi(f)(x) - \Phi(f)(y)| &= \bigg|\bigg(1 - \frac{\epsilon}{n}\bigg)f(x) - \bigg(1 - \frac{\epsilon}{n}\bigg)f(y)\bigg| = \bigg(1 - \frac{\epsilon}{n}\bigg)|f(x) - f(y)|\\
  &\leq (n - \epsilon)d(x,y).
 \end{align*}
It follows that $\lipc{\Phi(f)} \leq n - \epsilon$.
Because $\|f\|_p \leq n$ for all $f \in L(n)$,
 we have that
 $$\|\Phi(f)\|_p = \bigg\|\bigg(1 - \frac{\epsilon}{n}\bigg)f\bigg\|_p = \bigg(1 - \frac{\epsilon}{n}\bigg)\|f\|_p \leq n - \epsilon,$$
 and more that
 \begin{align*}
  \|\Phi(f) - f\|_p &= \bigg(\int_X |\Phi(f)(x) - f(x)|^p d\mu(x)\bigg)^{1/p} = \bigg(\int_X \bigg|\bigg(1 - \frac{\epsilon}{n}\bigg)f(x) - f(x)\bigg|^p d\mu(x)\bigg)^{1/p}\\
  &= \frac{\epsilon}{n}\bigg(\int_X |f(x)|^p d\mu(x)\bigg)^{1/p} \leq \frac{\epsilon}{n}\|f\|_p \leq \epsilon.
 \end{align*}
For each $\lambda \in (0,1)$, letting $\delta = n\lambda/(n - \epsilon)$, we have that if $\|f - g\|_p \leq \delta$,
 then
 $$\|\Phi(f) - \Phi(g)\|_p = \bigg\|\bigg(1 - \frac{\epsilon}{n}\bigg)f - \bigg(1 - \frac{\epsilon}{n}\bigg)g\bigg\|_p = \bigg(1 - \frac{\epsilon}{n}\bigg)\|f - g\|_p \leq \bigg(1 - \frac{\epsilon}{n}\bigg)\delta = \lambda,$$
 which implies the continuity of $\Phi$.
As a result, $\Phi$ is the desired map.

Next we will construct the map $\Psi$.
Due to the similar argument as the above, we may assume that there is $m < n$ such that $\|f\|_p \leq m$ and $\lipc{f} \leq m$ for every $f \in L(n)$, and that $\epsilon \leq n - m$.
Taking a point $y_0 \in X \setminus \{x_0\}$, let
 $$r_1 = \min\bigg\{\frac{\epsilon}{3n(\mu(B(x_0,n)))^{1/p}},d(x_0,y_0),n\bigg\}$$
 and $r_2 = (n - m)r_1/n$.
For each $f \in L(n)$, we define a function $\psi(f) : X \setminus (B(x_0,r_1) \setminus B(x_0,r_2)) \to \R$ by
 $$\psi(f)(x) = \left\{
 \begin{array}{ll}
  f(x_0) + n(r_2 - d(x,x_0)) &\text{if } x \in B(x_0,r_2),\\
  f(x) &\text{if otherwise}.
 \end{array}
 \right.$$
Verify that $\lipc{\psi(f)} = n$.
Indeed, when $x, y \in X \setminus B(x_0,r_1)$,
 $$|\psi(f)(x) - \psi(f)(y)| = |f(x) - f(y)| \leq md(x,y) < nd(x,y),$$
 and when $x \in B(x_0,r_2)$ and $y \in X \setminus B(x_0,r_1)$,
 \begin{align*}
  |\psi(f)(x) - \psi(f)(y)| &= |f(x_0) + n(r_2 - d(x,x_0)) - f(y)|\\
  &\leq |f(x_0) - f(x)| + |f(x) - f(y)| + n(r_2 - d(x,x_0))\\
  &\leq md(x,x_0) + md(x,y) + n(r_2 - d(x,x_0))\\
  &= md(x,y) + (n - m)(r_1 - d(x,x_0))\\
  &\leq md(x,y) + (n - m)(d(y,x_0) - d(x,x_0))\\
  &\leq md(x,y) + (n - m)d(x,y) = nd(x,y).
 \end{align*}
Since $\mu(\{x_0\}) = 0$ and $\mu(B(x_0,r_2)) > 0$,
 the cardinality of $B(x_0,r_2)$ is greater than or equal to $2$.
Then for any $x, y \in B(x_0,r_2)$,
\begin{align*}
 |\psi(f)(x) - \psi(f)(y)| &= |f(x_0) + n(r_2 - d(x,x_0)) - f(x_0) - n(r_2 - d(y,x_0))|\\
 &= n|d(y,x_0) - d(x,x_0)| \leq nd(x,y),
 \end{align*}
 and taking any $z \in B(x_0,r_2) \setminus \{x_0\}$, we have that
 $$|\psi(f)(x_0) - \psi(f)(z)| = |f(x_0) + n(r_2 - d(x_0,x_0)) - f(x_0) - n(r_2 - d(z,x_0))| = nd(z,x_0).$$
Hence $\lipc{\psi(f)} = n$.

According to McShane's lipschitz extension \cite{McS} (cf.~\cite[Theorem~13.12]{Ye}), the desired map $\Psi : L(n) \to L(n)$ can be defined as follows:
 $$\Psi(f)(x) = \left\{
 \begin{array}{ll}
  \psi(f)(x) &\text{if } x \in X \setminus (B(x_0,r_1) \setminus B(x_0,r_2)),\\
  \inf_{a \in X \setminus (B(x_0,r_1) \setminus B(x_0,r_2))}(\psi(f)(a) + nd(a,x)) &\text{if otherwise}.
 \end{array}
 \right.$$
Recall that $\lipc{\Psi(f)} = n$.
It immediately follows from the definition that $\Psi(f)(x) = f(x) = 0$ for all $x \in X \setminus B(x_0,n)$.
We shall prove that $\|\Psi(f) - f\|_p \leq \epsilon$ for every $f \in L(n)$.
Fix any $x \in B(x_0,r_2)$,
 so
 \begin{align*}
  |\Psi(f)(x) - f(x)| &= |f(x_0) + n(r_2 - d(x,x_0)) - f(x)| \leq |f(x_0) - f(x)| + n(r_2 - d(x,x_0))\\
  &\leq md(x,x_0) + n(r_2 - d(x,x_0)) \leq (m + n)r_2.
 \end{align*}
Taking each $x \in X \setminus B(x_0,r_1)$, we get that
 $$|\Psi(f)(x) - f(x)| = |f(x) - f(x)| = 0.$$
Let any $x \in B(x_0,r_1) \setminus B(x_0,r_2)$ and any $a \in X \setminus (B(x_0,r_1) \setminus B(x_0,r_2))$.
In the case that $a \in B(x_0,r_2)$,
\begin{align*}
 |\psi(f)(a) - f(x)| &= |f(x_0) + n(r_2 - d(a,x_0)) - f(x)| \leq |f(x_0) - f(x)| + n(r_2 - d(a,x_0))\\
 &\leq md(x,x_0) + n(r_2 - d(a,x_0)) \leq n(d(a,x) + r_2).
\end{align*}
In the case that $a \in X \setminus B(x_0,r_1)$,
 $$|\psi(f)(a) - f(x)| = |f(a) - f(x)| \leq md(a,x).$$
Therefore we have
 $$-nr_2 \leq \psi(f)(a) + nd(a,x) - f(x) \leq 2nd(a,x) + nr_2,$$
 which implies that
 \begin{align*}
  |\Psi(f)(x) - f(x)| &= \bigg|\inf_{a \in X \setminus (B(x_0,r_1) \setminus B(x_0,r_2))}(\psi(f)(a) + nd(a,x)) - f(x)\bigg|\\
  &\leq 2nd(x,x_0) + nr_2 \leq n(2r_1 + r_2) \leq 3nr_1.
 \end{align*}
It follows that
\begin{align*}
 \|\Psi(f) - f\|_p &= \bigg(\int_X |\Psi(f)(x) - f(x)|^p d\mu(x)\bigg)^{1/p} \leq \bigg(\int_{B(x_0,n)} (3nr_1)^p d\mu(x)\bigg)^{1/p}\\
 &= 3nr_1(\mu(B(x_0,n)))^{1/p} \leq \epsilon,
\end{align*}
 and that
 $$\|\Psi(f)\|_p \leq \|\Psi(f) - f\|_p + \|f\|_p \leq \epsilon + m \leq n.$$

To show the continuity of $\Psi$, fix any $f \in L(n)$ and any $\lambda > 0$.
Set $\sigma = \lambda/(\mu(B(x_0,n))^{1/p}$ and $\tau = \min\{\sigma/(8m),1\}$.
By the same argument as Lemma~\ref{ball}, we can choose $k \in \N$ so that $\gamma^{-k}\mu(B(x_0,n)) \leq \mu(B(x,\tau))$ for every $x \in B(x_0,n)$,
 where $\gamma$ be the doubling constant of $X$.
Let $\delta = \sigma\gamma^{-k/p}(\mu(B(x_0,n)))^{1/p}/4$.
Taking every $g \in L(n)$ such that $\|g - f\|_p \leq \delta$, we can obtain that for each $x \in B(x_0,n)$, $|f(x) - g(x)| \leq \sigma/2$.
Indeed, suppose not,
 so there exists $y \in B(x_0,n)$ such that $|f(y) - g(y)| > \sigma/2$.
Then for all $x \in B(y,\tau)$,
$$|f(x) - g(x)| \geq |f(y) - g(y)| - |f(x) - f(y)| - |g(x) - g(y)| > \frac{\sigma}{2} - 2md(x,y) \geq \frac{\sigma}{2} - 2m\tau \geq \frac{\sigma}{4}$$
 because $f$ and $g$ are $m$-lipschiz.
Thus
\begin{align*}
 \delta &\geq \|f - g\|_p \geq \bigg(\int_{B(y,\tau)} |f(x) - g(x)|^p d\mu(x)\bigg)^{1/p} > \frac{\sigma(\mu(B(y,\tau)))^{1/p}}{4}\\
 &\geq \frac{\sigma\gamma^{-k/p}(\mu(B(x_0,n)))^{1/p}}{4} = \delta,
\end{align*}
 which is a contradiction.
Hence for every $x \in B(x_0,r_2)$,
\begin{align*}
 |\Psi(f)(x) - \Psi(g)(x)| &= |f(x_0) + n(r_2 - d(x,x_0)) - g(x_0) - n(r_2 - d(x,x_0))|\\
 &= |f(x_0) - g(x_0)| \leq \frac{\sigma}{2}.
\end{align*}
For any $x \in B(x_0,r_1) \setminus B(x_0,r_2)$, we can find $b \in X \setminus (B(x_0,r_1) \setminus B(x_0,r_2))$ such that
$$\Psi(f)(x) = \inf_{a \in X \setminus (B(x_0,r_1) \setminus B(x_0,r_2))}(\psi(f)(a) + nd(a,x)) \geq \psi(f)(b) + nd(b,x) - \sigma/2.$$
Verify that if $b \in B(x_0,r_2)$,
 then
 \begin{align*}
  |\psi(f)(b) - \psi(g)(b)| &= |f(x_0) + n(r_2 - d(b,x_0)) - g(x_0) - n(r_2 - d(b,x_0))|\\
  &= |f(x_0) - g(x_0)| \leq \frac{\sigma}{2},
 \end{align*}
 and that if $b \in X \setminus B(x_0,r_1)$,
 then
 $$|\psi(f)(b) - \psi(g)(b)| = |f(b) - g(b)| \leq \frac{\sigma}{2}.$$
Thus we get
\begin{align*}
 \Psi(f)(x) &\geq \psi(f)(b) + nd(b,x) - \frac{\sigma}{2} \geq \psi(g)(b) + nd(b,x) - \sigma\\
 &\geq \inf_{a \in X \setminus (B(x_0,r_1) \setminus B(x_0,r_2))}(\psi(g)(a) + nd(a,x)) - \sigma = \Psi(g)(x) - \sigma.
\end{align*}
Similarly, $\Psi(g)(x) \geq \Psi(f)(x) - \sigma$,
 and hence $|\Psi(f)(x) - \Psi(g)(x)| \leq \sigma$.
Moreover, for each $x \in B(x_0,n) \setminus B(x_0,r_1)$,
 $$|\Psi(f)(x) - \Psi(g)(x)| = |f(x) - g(x)| \leq \frac{\sigma}{2}.$$
Therefore we have that
\begin{align*}
 \|\Psi(f) - \Psi(g)\|_p^p &= \int_{B(x_0,n)} |\Psi(f)(x) - \Psi(g)(x)|^p d\mu(x) + \int_{X \setminus B(x_0,n)} |\Psi(f)(x) - \Psi(g)(x)|^p d\mu(x)\\
 &= \int_{B(x_0,n)} \sigma^p d\mu(x) \leq \sigma^p\mu(B(x_0,n)) \leq \lambda^p,
\end{align*}
 so $\|\Psi(f) - \Psi(g)\|_p \leq \lambda$.
Consequently, $\Psi$ is continuous.
Thus the proof is completed.
\end{proof}

Now we show the following:

\begin{prop}\label{cvx.Q}
Let $X$ be a doubling metric measure space with $\mu(\{x_0\}) = 0$ such that for all $x \in X$ and $r > 0$, $\mu(B(x,r) \triangle B(y,r)) \to 0$ as $y$ tends to $x$.
Then for every $n \in \N$, the subset $L(n)$ is homeomorphic to $\Q$.
\end{prop}

\begin{proof}
By Proposition~\ref{sigma-cpt.}, each $L(n)$ is compact.
As is easily observed,
 $L(n)$ is a convex subset of the Banach space $\lp(X)$,
 and hence it is an AR, refer to \cite[Theorem~6.1.1]{Sakaik11}.
According to Lemma~\ref{dcp}, $L(n)$ satisfies condition~$(\ast)$ of Theorem~\ref{char.Q},
 so it is homeomorphic to $\Q$.
\end{proof}

\section{Proof of Corollary~\ref{lip.}}

Applying the following criterion, which was proven by D.~Curtis, T.~Dobrowolski and J.~Mogilski \cite{CDM},
 we shall establish Corollary~\ref{lip.}.

\begin{thm}\label{char.cvx.}
Let $C$ be a $\sigma$-compact convex set in a completely metrizable linear space $E$.
Suppose that the closure $\cl_E{C}$ is an AR and not locally compact.
Then the pair $(\cl_E{C},C)$ is homeomorphic to $(\ell_2 \times \Q,\ell_2^f \times \Q)$ if $C$ contains an infinite-dimensional locally compact convex set.
\end{thm}

We show the density of $\lip_b(X)$ in $\lp(X)$.

\begin{prop}\label{dense}
Let $X$ be a doubling metric measure space.
Suppose that for every point $x \in X$, the function $(0,\infty) \ni r \mapsto \mu(B(x,r)) \in (0,\infty)$ is continuous.
The subspace $\lip_b(X)$ is dense in $\lp(X)$.
\end{prop}

\begin{proof}
To show the density of $\lip_b(X) \subset \lp(X)$, fix any $f \in \lp(X)$ and any $\epsilon > 0$.
By virtue of Vitali's covering theorem, see \cite[Theorem~6.20]{Ye},
 we can choose a simple function $\sum_{i = 1}^n a_i\chi_{B(x_i,r_i)}$, where each $a_i \neq 0$, each $r_i > 0$ and $\{B(x_i,r_i) \mid i = 1, \cdots, n\}$ is pairwise disjoint,
 so that $\|\sum_{i = 1}^n a_i\chi_{B(x_i,r_i)} - f\|_p \leq \epsilon/3$.
Moreover, due to the assumption, for each $i \in \{1, \cdots, n\}$, there exists $\delta_i \in (0,r_i)$ such that $\|\sum_{i = 1}^n a_i\chi_{B(x_i,\delta_i)} - \sum_{i = 1}^n a_i\chi_{B(x_i,r_i)}\|_p \leq \epsilon/3$.
Then we can define a map $g : X \to \R$ as follows:
 $$g(x) = \left\{
 \begin{array}{ll}
  a_i &\text{if } x \in B(x_i,\delta_i), i = 1, \cdots, n,\\
  \frac{-a_i(d(x,x_i) - \delta_i)}{r_i - \delta_i} + a_i &\text{if } x \in B(x_i,r_i) \setminus B(x_i,\delta_i), i = 1, \cdots, n,\\
  0 &\text{if otherwise}.
 \end{array}
 \right.$$
As is easily observed,
 $g \in \lip_b(X)$.
Since
 $$\Bigg|g(x) - \sum_{i = 1}^n a_i\chi_{B(x_i,\delta_i)}(x)\Bigg| \leq \Bigg|\sum_{i = 1}^n a_i\chi_{B(x_i,\delta_i)}(x) - \sum_{i = 1}^n a_i\chi_{B(x_i,r_i)}(x)\Bigg|$$
 for every $x \in X$,
 \begin{align*}
  \Bigg\|g - \sum_{i = 1}^n a_i\chi_{B(x_i,\delta_i)}\Bigg\|_p &= \Bigg(\int_X \Bigg|g(x) - \sum_{i = 1}^n a_i\chi_{B(x_i,\delta_i)}(x)\Bigg|^p d\mu(x)\Bigg)^{1/p}\\
  &\leq \Bigg(\int_X \Bigg|\sum_{i = 1}^n a_i\chi_{B(x_i,\delta_i)}(x) - \sum_{i = 1}^n a_i\chi_{B(x_i,r_i)}(x)\Bigg|^p d\mu(x)\Bigg)^{1/p}\\
  &= \Bigg\|\sum_{i = 1}^n a_i\chi_{B(x_i,\delta_i)} - \sum_{i = 1}^n a_i\chi_{B(x_i,r_i)}\Bigg\|_p \leq \frac{\epsilon}{3},
 \end{align*}
 we get that
 $$\|g - f\|_p \leq \Bigg\|g - \sum_{i = 1}^n a_i\chi_{B(x_i,\delta_i)}\Bigg\|_p + \Bigg\|\sum_{i = 1}^n a_i\chi_{B(x_i,\delta_i)} - \sum_{i = 1}^n a_i\chi_{B(x_i,r_i)}\Bigg\|_p + \Bigg\|\sum_{i = 1}^n a_i\chi_{B(x_i,r_i)} - f\Bigg\|_p \leq \epsilon.$$
Thus the proof is completed.
\end{proof}

Remark that for any $x, y \in X$ and any $0 < \delta < r$, if $d(x,y) < \delta$,
 then
 $$B(x,r) \triangle B(y,r) \subset B(x,r + \delta) \setminus B(x,r - \delta).$$
Hence as $y$ tends to $x$,
 $$\mu(B(x,r) \triangle B(y,r)) \leq \mu(B(x,r + \delta) \setminus B(x,r - \delta)) \to 0$$
 when $(0,\infty) \ni r \mapsto \mu(B(x,r)) \in (0,\infty)$ is continuous.
Now we shall prove Corollary~\ref{lip.}.

\begin{proof}[Proof of Corollary~\ref{lip.}]
Since $X$ is doubling and the function $(0,\infty) \ni r \mapsto \mu(B(x,r)) \in (0,\infty)$ is continuous for each fixed point $x \in X$,
 we have that $\mu(\{y \in X \mid d(x,y) = r\}) = 0$ for all $x \in X$ and $r > 0$ by \cite[Lemma~8.5]{Ye}.
Moreover, because $X$ is non-degenerate,
 $\mu(\{x\}) = 0$ for every $x \in X$,
 in particular $\mu(\{x_0\}) = 0$ for the fixed point $x_0 \in X$.
Additionally, every open ball with positive radius is of positive measure,
 so $X$ has no isolated points,
 which implies that $X$ is infinite.
Since $X$ is separable,
 the space $\lp(X)$ is homeomorphic to $\ell_2$ by Theorem~\ref{Lp}.
Observe that for each $x \in X$ and each $r > 0$, $\mu(B(x,r) \triangle B(y,r)) \to 0$ as $y \to x$.
By Proposition~\ref{sigma-cpt.}, the subspace $\lip_b(X)$ is $\sigma$-compact.
As is easily observed,
 $\lip_b(X)$ is convex in $\lp(X)$.
Due to Proposition~\ref{cvx.Q}, $\lip_b(X)$ contains the infinite-dimensional compact convex set $L(n)$ which is homeomorphic to $\Q$.
According to Proposition~\ref{dense}, $\lip_b(X)$ is dense in $\lp(X)$.
Applying Theorem~\ref{char.cvx.}, the pair $(\lp(X),\lip_b(X))$ is homeomorphic to $(\ell_2 \times \Q,\ell_2^f \times \Q)$.
\end{proof}

\end{document}